\documentclass{amsart}
\usepackage{pdfsync}
\usepackage{xspace}
\usepackage[all]{xy}
\usepackage{enumerate}
\usepackage{amssymb}
\usepackage{amsmath}
\usepackage{tikz-cd}

\makeatletter
\DeclareRobustCommand*\cal{\@fontswitch\relax\mathcal}
\makeatother

\newtheorem{thm}{Theorem}
\newtheorem{definition}{Definition}
\newtheorem{notation}[definition]{Notation}
\newtheorem{rem}[definition]{Remark}

\newtheorem{lem}[definition]{Lemma}
\newtheorem{cor}[definition]{Corollary}
\newtheorem{fact}[definition]{Fact}
\newtheorem{quest}{Question}

\newcommand{\br}{\overline}
\newcommand{\A}{\mathcal A}

\newcommand{\B}{\mathcal B}
\newcommand{\C}{\mathcal C}

\newcommand{\G}{\mathcal G}
\newcommand{\K}{\mathcal K}

\newcommand{\free}{F_{\aleph_0}}
\newcommand{\Tfree}{T_{\aleph_0}}
\newcommand{\pfree}{F_{\Rmod}}
\newcommand{\Tpfree}{T_{\Rmod}}

\newcommand{\pp}{\mathop{\rm pp}}

\newcommand{\RMod}{R\textrm{-Mod}}
\newcommand{\Rmod}{R\textrm{-mod}}

\newcommand{\Add}{\mathop{{\rm Add}}\nolimits}
\newcommand{\add}{\mathop{{\rm add}}\nolimits}

\newcommand{\Th}{\mathop{{\rm Th}}\nolimits}

\renewcommand{\phi}{\varphi}

\makeatletter
\@namedef{subjclassname@2020}{\textup{2020} Mathematics Subject Classification}
\makeatother

\begin{document}
\title{%Universality of free modules  and Bass modules\\
Bass modules and embeddings into free modules}
\author{Anand Pillay and Philipp Rothmaler}
\date{\today}
\maketitle
\begin{abstract} We show that the free module of infinite rank $R^{(\kappa)}$ purely embeds every $\kappa$-generated flat left $R$-module iff $R$ is left perfect. Using a Bass module corresponding to a descending chain of principal right ideals, we construct a model of the theory $T$ of $R^{(\kappa)}$ whose projectivity is equivalent to left perfectness, which allows to add a `stronger' equivalent condition: $R^{(\kappa)}$ purely embeds every $\kappa$-generated flat left $R$-module which is a model of $T$. %???do we really have it for $\kappa$???

We extend the model-theoretic construction of this Bass module to arbitrary descending chains of pp formulas, resulting in a `Bass theory' of pure-projective modules. We put this new theory to use by, among other things, reproving an old result of Daniel Simson about pure-semisimple rings and Mittag-Leffler modules.

This paper is a condensed version, solely about modules, of the larger work \cite{arxiv} with two new results added about cyclically presented modules (Cor.14) and finitely presented cyclic modules (Rem.15).
\end{abstract}

\footnotetext{The first author was partially supported by NSF grants DMS-1665035,
DMS-1760212, and DMS-2054271.}

\section{Introduction} 
Bass \cite{Ba} showed how an infinite descending chain of principal right ideals gives rise to a countably generated flat left module that is not projective---a so-called \emph{Bass module}. \cite{PR} provided a model-theoretic proof of this, which we extend here to construct 
 \emph{generalized Bass modules}---countable direct limits that are not pure-projective---for any given descending chain of positive primitive (pp) formulas, Thm.\,\ref{M}. %and  corresponding generalized Bass modules that are not pure-projective,
Part of Bass' characterization of one-sided perfect rings follows as a special case, Cor.\,\ref{BassThm}. 

Another novelty is that, in case of Bass' original result, we obtain such a Bass \emph{model}  of the first-order theory of the free modules of infinite rank.  In the general case we find a generalized Bass module which is a model of the  first-order theory of a direct sum of certain finitely presented (f.p.) modules, modules that we call pure-free for their analogy with free modules, where instead of just $_RR$ we allow arbitrary (left) f.p. modules in the construction. These, in turn, are chosen to be f.p.\ free realizations of the pp formulas in question, and so we arrive at a specific class of f.p.\ modules for each choice of descending chain of pp formulas.

%We extend Bass' limit construction of a non-projective module over a non-perfect ring to pure-projective modules, of which Bass' theorem on one-sided perfect rings is a special case, Cor.\,\ref{Cor1}. 

These results grew out of model-theoretic investigations about categoricity, saturation and universality of free algebras of infinite rank in varieties in the sense of universal algebra, see \cite{arxiv} for the larger story. May it just be mentioned that while \texttt{universal} structures are the focus there, here Cor.\ \ref{BassThm}(v)--(vi) (plus extra clause) could be stated as saying that a ring is left perfect if and only if the free module of rank $\kappa$ is universal among the $\kappa$-generated flat modules with respect to pure  embeddings if and only if it is  universal  among the $\kappa$-generated flat models of its own theory with respect to elementary embeddings. Items (v) and (v)$_T$ of the theorem and the other corollaries are generalizations of this universality result to the corresponding pure-free modules or models.

%in theorem and corollaries below state that (certain pure-) free modules are universal with respect to pure (or elementary) embedding among the countably generated flat modules
%
%
%\begin{remquest}
% about \texttt{universal}
%\end{remquest}
%

The second author would like to thank Martin Ziegler for straightening out item (b) in the final remark of the paper.

\section{Preliminaries} All modules are unitary left $R$ modules over an associative ring $R$ with $1$. The class of all such is denoted by $\RMod$. The shorthand
f.g.\ means \emph{finitely generated}, f.p.\ means \emph{finitely presented}.  $\Rmod$ denotes the (skeletally small) class of all f.p.\ modules in $\RMod$. Given a class $\K\subseteq\RMod$, the classes of all direct summands of all direct sums (resp., all \emph{finite} direct sums) of modules from $\K$ is denoted by  $\Add(\K)$ (resp., $\add(\K)$). 
Note, the \texttt{projective modules} are precisely the modules in $\Add(_RR)$ while the f.g.\ projectives are the modules in  $\add(_RR)$. Correspondingly, the  \texttt{pure-projective modules} are precisely the modules in $\Add(\Rmod)$ while the f.g.\ pure-projectives are the modules in $\add(\Rmod)$. In other words, the projective modules are the direct summands of free modules, while the pure-projective modules are the direct summands of direct sums of f.p.\ modules, which is why, in analogy, we call direct sums of f.p.\ modules \texttt{pure-free}. 

It is well known that the pure-projective modules are the modules that are projective w.r.t.\ \texttt{pure-exact sequences}, where a monomorphism is said to be \texttt{pure}  if it preserves the pp type of all tuples (elements suffice!). 

Throughout, \texttt{tuples} are finite sequences of elements---for left modules usually thought of as column vectors. (The same applies to tuples of variables.) The pp type of a tuple $\bar a$ in a module $M\in\RMod$ is the set $\pp_M(\bar a)$ of all pp formulas realized by $\bar a$ in $M$. Such a pp type is said to be \texttt{finitely generated} or \texttt{f.g.}\/ if it contains a formula $\phi$ which implies, in \emph{every} module, every other formula in it, i.e., $\phi\leq\psi$ for every $\psi\in\pp_M(\bar a)$, where $\leq$ is the preordering of implication in the  lattice of pp formulas of  corresponding arity in $\RMod$, i.e., $\phi\leq\psi$ iff $\phi(M) \subseteq\psi(M)$ for \emph{every} $M\in\RMod$.
(Here we have to assume basic familiarity with pp formulas. Recall,  a typical $n$-ary pp formula is $A|B\bar x$ or, in its long form, $\exists\bar y A\bar y \dot= B\bar x$, where $A$ and $B$ are matching matrices over $R$, i.e., $A$ is $m\times l(\bar y)$ and $B$ is $m\times l(\bar x)=m\times n$. The reader  is referred to  \cite{P1} and  \cite{P2} for more detail.)

\texttt{Mittag-Leffler modules} are characterized as the modules in which every tuple has a finitely generated (f.g.) pp type, \cite[Main Thm.]{habil}, \cite{R}, \cite{R2} or \cite{P2}. While pure-projective modules are Mittag-Leffler, they enjoy a stronger property: if $M$ is pure-projective, every tuple $\bar a$ in $M$ satisfies a certain pp formula $\phi$ \texttt{freely}, which means that for every module $N\in\RMod$ and every tuple $\bar b$ that satisfies $\phi$  therein, there is a map $h: M\to N$ sending $\bar a$ to $\bar b$. We write $h: (M, \bar a) \to (N, \bar b)$ (thinking of this map as a morphism of pointed modules) and call $(M, \bar a)$ a \texttt{free realization} of $\phi$. It is easily seen that then $\phi$ generates $\pp_M(\bar a)$ (so $M$ is indeed Mittag-Leffler).

It is a classical result from  \cite[2.2.2]{RG} that countably generated Mittag-Leffler modules are pure-projective, see  \cite[Lemma 3.9]{habil} for a model-theoretic proof or \cite[Thm.\,1.3.26]{P2}.

For later reference we summarize.

\begin{fact} \label{ML}%\label{lem}
% Let $\kappa$ be an infinite cardinal and $\K\subseteq \C$ be a class of $\kappa$-generated modules.  
\begin{enumerate}[\rm(1)]
 \item Every pure submodule of a Mittag-Leffler module is Mittag-Leffler.
 \item Every pure-free module is pure-projective (hence Mittag-Leffler).
 \item The pure-projective modules are precisely the direct summands of the pure-free modules (just as the projectives are the direct summands of the free modules).
 \item A module  is  pure-projective (resp., projective) if and only if it is in $\Add(\Rmod)$ (resp., in $\Add({_RR})$).
 \item Every pure-projective module is Mittag-Leffler.
 \item The converse holds for countably generated modules.
 \item Hence every countably generated pure submodule of a Mittag-Leffler module is pure-projective.
 %  \item A module $K\in \K$ is pure-projective if and only if it is (isomorphic to) a direct summand of $F_\A^{(\kappa)}$.
% \item Any pure-projective model of $T_\A$ in $\K$ is elementarily embedded in $F_\A^{(\kappa)}$ (in fact, as a direct summand).
%\item Every pure submodule of $F_\A^{(\kappa)}$ is Mittag-Leffler.
\end{enumerate}
\end{fact}

\begin{notation}\label{gen}
 Given a class of modules $\G$, denote by $\Gamma_{\G}$ the set of all pp formulas (of any arity) that generate a pp type in a module from $\G$.  
 
 Clearly, if $\G\subseteq\Add(\Rmod)$, then $\Gamma_{\G}$ is the set of all pp formulas that some pointed module from $\G$ freely realizes.
\end{notation}

Given any class $\K\subseteq\RMod$, the sets of pp types in $\add(\K)$ and in $\Add(\K)$ are the same---for two reasons. One, the pp type of a tuple in a direct summand $A$ of $B$ is the same in $A$ as in $B$, for the simple reason that direct summands are pure. Two, every tuple in an infinite direct sum of modules from $\K$ is contained already in a finite subsum, hence in a certain module from $\add(\K)$.

As customary in the model theory of modules, we consider modules as first-order structures in an elementary (= finitary first-order) language $L$ that has, besides a constant symbol $0$ and a binary operation symbol $+$, unary function symbols for every scalar from the ring (and therefore depends on the given ring). The $L$-theory, $\Th(\K)$ of a class of modules $\K$   is the set of all $L$-sentences that are true in all members of $\K$. Given a single module $M$, one writes $\Th(M)$ instead of $\Th(\{M\})$ and calls this object \texttt{the complete theory of $M$}. Two modules are said to be \texttt{elementarily equivalent} if they possess the same complete theory. An arbitrary (elementary) theory is called \texttt{complete} if all of its models are elementarily equivalent. 

The following is well known---the second part is an easy consequence of the first, which in turn follows from the well-known pp-elimination for modules. A \texttt{pp index} in a module $M$ is the the size of the factor group $\phi(M)/\psi(M)$ in case it is finite, and the symbol $\infty$ otherwise. Here $\phi$ and $\psi$ are unary pp formulas with 
$\psi\leq\phi$ (which means, remember, that $\psi(N)\subseteq\phi(N)$ for all $N\in\RMod$).

\begin{fact}
\begin{enumerate}[\rm(1)]
 \item Two modules  (over a given ring)  are elementarily equivalent if and only if they have the same pp indices. 
  \item The theory of all free modules of infinite rank is complete, i.e., all free modules of infinite rank (over a given ring) are elementarily equivalent.
\end{enumerate}
 \end{fact}

Therefore, the theory of all free modules of infinite rank is the (complete) theory $\Th(\free)$ of the free module $\free= {_RR}^{(\omega)}$ of rank $\aleph_0$. We denote it by $\Tfree$.

% an $n$-tuple in a module $M$ is an $n\times 1$ vector over $M$, which allows for left scalar multiplication by  $m\times n$-matrices over $R$. If $\bar a$ is such a tuple, we write $l(\bar a)=n$.  Similarly, a variable tuple $\bar x$ is an $l(\bar x)\times 1$ column vector (of variables), which allows for left scalar multiplication by  $m\times l(\bar x)$-matrices over $R$. 

\section{Generalized Bass modules}
 We generalize to arbitrary descending chains of pp formulas Bass' construction of a flat module that is not projective when the ring has an infinite descending chain of principal right ideals.%\footnote{T.\,Y.\,Lam says:
%\emph{This switch
%from left modules to right modules, albeit not new for Bass ..., is
%in fact one of the inherent peculiar features of his Theorem...\,. Unfortunately,
%because of this unusual switch of sides, [the] Theorem  is often misquoted in the
%literature, sometimes even in authoritative sources;...\,\cite[p.\,24]{Lam}.} From the model-theoretic perspective, or likewise in the terminology of p-functors of \cite{Zim}, there is nothing unusual about this switch. In fact, there is none, if we replace \emph{right} principal ideals by  \emph{left} finitary matrix subgroups or  pp formulas  that define them---which is the thing to do as the  theorem shows. The only switch of sides then is by the (rather accidental) fact that a left pp formula defines a right ideal---in \emph{any} ring, nothing  special about perfect rings.}
\begin{definition}\label{Bass} Suppose $\Phi$ is  a descending chain  of pp formulas of fixed arity, $\phi_0\geq\phi_1\geq\phi_2\geq\ldots$.
%\begin{enumerate}
% \item 
 
 A \texttt{Bass module $B_\Phi$} is the direct limit of a  direct system (in fact, a chain) obtained as follows.
Choose finitely presented free realizations $(A_i, \br a_i)$ of $\phi_i$ and maps $g_i : (A_i, \br a_i) \to (A_{i+1}, \br a_{i+1})$ for all $i$, which exist because $ \br a_{i+1}$ satisfies $\phi_i$ in $A_{i+1}$. (This choice is by no means unique.)

Consider also the corresponding maps  $f_i : A_{i}\to B_\Phi$ and the module $F_\Phi := \bigoplus_i A_i^{(\omega)}$,  the direct sum of infinitely many copies of each of the $A_i$, $i<\omega$.

% \item ???
%\end{enumerate}
  \end{definition}

\begin{lem}\label{Bassversion}
If the chain $\Phi$  does not stabilize  (uniformly in $\RMod$), the module $B_\Phi$ is not Mittag-Leffler, hence not pure-projective either.
%If the module $B_\Phi$ is pure-projective (even just Mittag-Leffler), the chain $\Phi$  must stabilize (uniformly in $\RMod$).
\end{lem}
\begin{proof}
If $\Phi$  does not stabilize, by \cite[Lemma 3.6]{PR}, the pp type of $f_i(\br a_i)$ in $B_\Phi$  is not finitely generated, hence $B_\Phi$ is not Mittag-Leffler.
\end{proof}

This suggests the significance of descending chain conditions (dcc) of the following kind.

\begin{definition}
Let $\G, \K \subseteq\RMod$ be classes of modules. 

$\K$ is said to have the  \texttt{$\Gamma_\G$-dcc} if for every $K\in\K$ and every descending chain $\gamma_0(\bar x)\geq \gamma_1(\bar x)\geq\gamma_2(\bar x)\geq\ldots$  of formulas  from $\Gamma_\G$ (in the sense of Notation \ref{gen} above)  of the same arity, the corresponding descending chain of subgroups of $K^{\l(\bar x)}$ defined by the $\gamma_i$ in $K$ stabilizes.
\end{definition}

\begin{lem} Let $\G\subseteq\Add(\Rmod)$.

Then $\RMod$ has the $\Gamma_\G$-dcc if and only if $\Add(\G)$ does.
\end{lem}
\begin{proof}
 For the nontrivial direction, consider a  descending chain $\gamma_0(\bar x)\geq \gamma_1(\bar x)\geq\gamma_2(\bar x)\geq\ldots$  from $\Gamma_\G$ and assume it does not stabilize in some module $M$. We have to produce a module from $\Add(\G)$ in which it does not stabilize either. To this end, first choose tuples $\bar a_i\in \gamma_i(M)\setminus\gamma_{i+1}(M)$, then free realizations $(G_i, \bar g_i)$ of $\gamma_i$ in $\Gamma_\G$, and finally maps $(G_i, \bar g_i)\to (M, \bar a_i)$, for all $i$.
 
 As pp formulas  are preserved by homomorphisms, also $\bar g_i\in \gamma_i(G_i)\setminus\gamma_{i+1}(G_i)$, for all $i$, hence this chain does not stabilize in $\bigoplus_{i<\omega} G_i\in \Add(\G)$.
\end{proof}

\section{The  setting}
We start from an arbitrary \emph{set}\footnote{This is no restriction, since $\Rmod$ is skeletally small.} of finitely presented modules, $\A\subseteq\Rmod$, 
%however we make sure it is a set, which is possibly by choosing a transversal, as $\Rmod$ is skeletally small, which is exactly that: the isomorphism classes of $\Rmod$ form a set. And so, throughout, we let $\A$ be a sub\emph{set} of $\Rmod$, 
close it under finite direct sums and direct summands, 
$\B=\add\A$, and let $\C=\lim\B$, the class of all direct limits (colimits) of modules from $\B$. Whenever we write one of theses letters, $\A$, $\B$, or $\C$, we tacitly associate the other two as just indicated. With any such choice we associate the following objects and concepts. %??? Maybe $\Gamma_\B$ is better, then also $\Gamma_F$ makes sense???
 
\begin{definition}
 
\begin{enumerate}[\rm(a)]
\item  $F_\A$ is the pure-free module $\bigoplus_{A\in\A} A^{(\omega)}$.
\item $T_\A$ is the (complete) $L$-theory of $F_\A$.
\item A \texttt{$\C$-model}  is a model of $T_\A$ that is at the same time a member of\/ $\C$.
% \item $\Gamma_\B$ is the set of all pp formulas of any finite arity that are freely realized in a module from $\B=\add\A$. ??? Should we call this $\Gamma_\B$???Then easier referal to other types.???
%\item A class $\K\subseteq\RMod$ has the \texttt{$\Gamma_\B$-dcc} if for every $K\in\K$ and every descending chain $\gamma_0(\bar x)\geq \gamma_1(\bar x)\geq\gamma_2(\bar x)\geq\ldots$  of formulas  of the same arity from $\Gamma_\B$ (in the sense of Notation \ref{gen} above), the corresponding descending chain of subgroups of $K^{\l(\bar x)}$ defined by the $\gamma_i$ in $K$ stabilizes.
\end{enumerate}
\end{definition}

%In the next lemma we gather, for ease of reference, a few facts, the first three of which are well-known properties of Mittag-Leffler modules.

\begin{lem} \label{lem}
 Let $\kappa$ be an infinite cardinal and $\K\subseteq \C$ be a class of $\kappa$-generated modules in $\C$.  
\begin{enumerate}[\rm(1)]
% \item Every pure submodule of $F_\A^{(\kappa)}$ is Mittag-Leffler.
 % \item Every countably generated pure submodule of $F_\A^{(\kappa)}$ is pure-projective.
   \item A module $K\in \K$ is pure-projective if and only if it is (isomorphic to) a direct summand of $F_\A^{(\kappa)}$.

 \item Any pure-projective model of\/ $T_\A$ in $\K$ is elementarily embedded in $F_\A^{(\kappa)}$ (in fact, as a direct summand).
%\item Every pure submodule of $F_\A^{(\kappa)}$ is Mittag-Leffler.
\end{enumerate}
\end{lem}
\begin{proof} (1) Pure-projective modules are precisely the direct summands of pure-free modules. To see that  for that pure-free one can take $F_\A^{(\kappa)}$ when 
$K\in\K\subseteq \C$, we invoke an old result of Lenzing \cite{Len}, see \cite[33.9(2)]{W} or \cite[Lemma 2.13]{GT}: any $C\in\C$ is the image of a pure-epimorphism  from a direct sum of modules from $\B$, which must split in case $C$ is pure-projective, whence in that case $C$ is a direct summand of a direct sum of modules from $\B$, hence also of  $F_\A^{(\kappa)}$, for every $B\in \B$ occurs as a direct summand of  $F_\A$ and one needs at most $\kappa$ summands for a $\kappa$-generated submodule.

To get from (2) to (1), all one has to realize is that, by a classical result of Sabbagh \cite{Sab}, the elementary embeddings of modules are exactly the pure embeddings between elementarily equivalent modules, cf.\ \cite[Prop.2.25]{P1} (and that $T_\A$ is complete, i.e., all of its models are elementarily equivalent).
\end{proof}

\section{Main results}
First we make Bass \emph{models} out of  Bass modules. We work in the above setting. %$T_\A$. 

\begin{lem}\label{flatmodel} Suppose $\Phi$ is a descending chain in $\Gamma_\B$ and $B_\Phi$ is the corresponding Bass module (Def.\ \ref{Bass}).
 \begin{enumerate}[\rm(1)]
 \item $F_\A\oplus B_\Phi$ is a model of $T_\A$.
\item If the chain $\Phi$  does not stabilize, $F_\A\oplus B_\Phi$  is a $\C$-model that is not Mittag-Leffler, hence not pure-projective either.
 \end{enumerate}
 % $R$ is left perfect if and only if every (countably generated)  flat model of $T$ is projective %(if and only if the model $F\oplus B$ is projective for $F$ the free module of countably infinite rank).
% (if and only if the model $R^{(\omega)}\oplus B$ ???what is $B$??? is projective). 
\end{lem}
\begin{proof} Obviously, $B_\Phi$ and $F_\A$ are in $\C$, hence so is their direct sum. Therefore (2) follows from (1) and Lemma \ref{Bassversion}. 

To prove (1) we verify that $F_\A$ and $M:= F_\A\oplus B_\Phi$  are elementarily equivalent. Both of these modules clearly have all pp indices infinite. So it suffices to prove that a pp pair $\phi/\psi$ opens up in one of them iff it does so in the other.
%
% the lemma, it  suffices to produce a flat model of\/ $T$ that is projective if and only if  $B$ is.
%We claim, $M := F \oplus B$ is such a model, where $F= R^{(\omega)}$. First of all, $M$ is flat and it is projective iff $B$ is. So it remains to see  $M\models T$.
% 
% Since all the pp indices $|\phi/\psi|$ are infinite in $F$, it suffices to verify that a pp pair $\phi/\psi$ opens up in $M$ if and only if it does in $F$. 
For the nontrivial implication of these, suppose it opens up in $M$ because it opens up in $B_\Phi$. It suffices to prove that it also opens up in $F_\A$.

Recall that $B_\Phi$ is a direct limit of modules from $\B$. By properties of direct limits, a pp pair that opens up in $B_\Phi$ must  open up in some $B\in\B$. But $B$ is a direct summand of some finite direct sum of modules from $\A$. So the  pair in question must open up in  a finite direct sum of modules---hence also    in some individual module---from $A\in \A$, thus also in $F_\A$, as desired.
%
%. , for Being a direct limit of modules from $\B$
%
%Then it has to open up in $_RR$ (a fancier way of saying this is that definable subcategories---which are defined by the closing of certain pp pairs---are closed under direct limit). But if it opens up in $R$, it does so in $F$ as well, and we are done.
\end{proof}

\begin{rem}\label{rem2}
 If the Bass module $B_\Phi$ \emph{is} pure-projective,  by Eilenberg's trick, $M:= F_\A\oplus B_\Phi\cong F_\A$, and so, trivially,  $M$ is a model of $T_\A$. The point of the above argument is that it is a model---whether $B_\Phi$ is pure-projective or not.
\end{rem}

%\begin{lem}\label (In the notation of Def.\ref{Bass}.)
%\begin{enumerate}[\rm (1)]
%\item 
%  $F_\Phi \oplus B_\Phi$ is a model of $T_\A$. 
%\item   If $\Phi$ does not stabilize, $B_\Phi$ is not Mittag-Leffler, hence not pure-projective.
%\qed
% \end{enumerate}
%\end{lem}
%\begin{proof}
% The (pure-projective) module $F_\Phi$ is elementarily equivalent to $F_\Phi \oplus B_\Phi$. 
%\end{proof}
%
%??? We are now able to state the universality properties of large enough pure-free modules.% $F_\A = \bigoplus_{A\in\A} A^{(\omega)}$ has, for any given choice of $\A\subseteq \Rmod$ (and $\B=\add\A$ and $\C=\lim\B$).
\begin{thm}\label{M} 
  The  following are equivalent for any ring $R$ and any given choice of $\A\subseteq \Rmod$, $\B=\add\A$ and $\C=\lim\B$.

\begin{enumerate}[\rm (i)]%\label{thm}
\item\label{i} $\RMod$ has the $\Gamma_\B$-dcc.
\item\label{ii} All modules in $\C$ are Mittag-Leffler.
 \setcounter{enumi}{1} 
 \item{\hspace{-0.5em}}$_T$ All $\C$-models of $T_\A$  are Mittag-Leffler.
\item\label{iii}  All countably generated modules in   $\C$ are pure-projective.
 \setcounter{enumi}{2} 
 \item{\hspace{-0.5em}}$_T$ All countably generated $\C$-models of $T_\A$  are  pure-projective.
  \item\label{iv}  Every countably generated module in   $\C$ is isomorphic to a direct summand of $F_\A$.
 \setcounter{enumi}{3} 
 \item{\hspace{-0.5em}}$_T$ Every countably generated $\C$-model of $T_\A$  is isomorphic to a direct summand of $F_\A$ that is an elementary substructure of $F_\A$.

 \item\label{v}  Every countably generated module in   $\C$ is purely embedded in $F_\A$.
 \setcounter{enumi}{4} 
 \item{\hspace{-0.5em}}$_T$ Every countably generated $\C$-model of $T_\A$   is elementarily embedded in $F_\A$.
 %\item\label{(4ii)}  $R^{(\kappa)}$ is pure-universal among $\kappa$-generated flat $R$-modules and  $\Sigma$-pure-injective, for some (every) $\kappa\geq\omega$. 
%\item All flat $R$-modules are $\Sigma$-pure-injective.
%\item\label{(4iv)}   $T$ is totally transcendental.
%\item\label{(4i)}  $R$ is left $\Sigma$-pure-injective.
%
\end{enumerate}
\end{thm}

\begin{proof}
(\ref{i})$\implies$(\ref{ii}): every $C\in\C$ is a direct limit of some of the $B_i$ in $\B$. To show $C$ is Mittag-Leffler, we verify that all pp types of tuples in $C$ are f.g.\ in the sense of \cite[Thm.\,2.2]{R} or \cite[Thm.\,1.3.22]{P2}. As $\B\subseteq \Rmod$, we may invoke \cite[Lemma 3.6]{PR}, which implies that all $C\in\C$ are Mittag-Leffler provided $\RMod$ has the $\Gamma_\B$-dcc. (To be precise,  in \cite{PR} this is stated for countable limits, but it is obvious that it applies to arbitrary limits, which was made explicit in \cite[Lemma 3.8]{R2}.)

(\ref{ii})$\implies$(\ref{iii}) by the aforementioned  classical result  of Raynaud and Gruson. 

(\ref{iii})$\implies$(\ref{iv}) and (\ref{v})$\implies$(\ref{iii}) follow from Lemma \ref{lem} above for $\kappa=\aleph_0$, while
(\ref{iv})$\implies$(\ref{v}) and (\ref{iii})$\implies$(\ref{iii})$_T$ are trivial.

%(\ref{v})$\implies$(\ref{iii}): Being a pure submodule of the Mittag-Leffler module $F_\A$, every countably generated module in $\C$ is Mittag-Leffler, hence pure-projective as before.

%(\ref{iii})$\implies$(\ref{iii})$_T$  is trivial.

Finally, (\ref{iii})$_T$$\implies$(\ref{i}) is (the contrapositive of) Lemma \ref{Bassversion},  which concludes the proof of equivalence of (i) through (v) and (\ref{iii})$_T$.

%\noindent
%We have proved the equivalence of (i) through (v) and (\ref{iii})$_T$.

By (the proof of) Lemma \ref{lem}(4), `elementary embeddings' in (\ref{iv})$_T$ and (\ref{v})$_T$ can be replaced by `pure embeddings'. 
Thus the string of implications 
(\ref{ii})$_T$$\implies\dots \implies$(\ref{v})$_T\implies$(\ref{iii})$_T$ follows in the same fashion as their  unsubscripted counterpart.

The missing link (ii)$\implies$(ii)$_T$ being trivial, this completes the proof. 
\end{proof}

%The implications (n)$\implies$(n)$_T$ are trivial for $n=ii$ and $n=iii$. For $n=iv$ and $n=v$ they follow from the fact that a pure embedding of elementarily equivalent modules  is elementary, Prest 1????, and the fact that all models  of $T_\A$ are elementarily equivalent because of the completeness of that theory.???
%
%Finally, the implications (ii)$_T$ through (v)$_T$ and (v)$_T$$\implies$(iii)$_T$, follow by the same arguments as their unsubscripted counterparts, which completes the proof.
%

First we apply the theorem to the largest possible class $\A$, that is $\A=\Rmod$. Then $\B=\A$ and $\C=\RMod$. Correspondingly, we obtain  the largest countably generated pure-free module $\pfree =\bigoplus_{A\in\Rmod} A^{(\omega)}$ and its elementary theory $\Tpfree$.

%\begin{notation}
%%Write $F^*$ 
%%for
%%
%When $\A = \Rmod$ is largest possible, we obtain  the "largest" countably generated pure-free module $\pfree =\bigoplus_{A\in\Rmod} A^{(\omega)}$ and its elementary theory $\Tpfree$.
%% When $\A=\Rmod$, we write $F*$ for the "largest" pure-free module 
%%MAybe just let it be $F_{\Rmod}$ and $T_{\Rmod}$. 
%We write $\free$ instead of $F_{\{_RR\}}$, for that's what it is, the free module of rank $\aleph_0$. Remember, its theory is denoted by    $\Tfree$. 
%\end{notation}

The equivalences (i)$\Leftrightarrow$(ii)$\Leftrightarrow$(iii) below were discovered by Daniel Simson, \cite{Sim}. 
 The equivalence (i)$\Leftrightarrow$(vi) goes back to a result of Prest (predating his books), \cite[Thm.\,4.5.1]{P2}.

\begin{cor}\label{Sim}
The following are equivalent for any ring $R$.

%  The  following are equivalent for any ring $R$ and any given choice of $\A\subseteq \Rmod$, $\B=\add\A$ and $\C=\lim\B$.

\begin{enumerate}[\rm (i)]%\label{thm}
\item\label{i} $\RMod$ has the dcc on all pp formulas (equivalently, the dcc on all unary pp formulas).
\item\label{ii} All modules are Mittag-Leffler.
 \setcounter{enumi}{1} 
 \item{\hspace{-0.5em}}$_T$ All models of $\Tpfree$  are Mittag-Leffler.
\item\label{iii}  All countably generated modules  are pure-projective.
 \setcounter{enumi}{2} 
 \item{\hspace{-0.5em}}$_T$ All countably generated models of $\Tpfree$  are  pure-projective.
  \item\label{iv}  Every countably generated module is isomorphic to a direct summand  of $\pfree$.
 \setcounter{enumi}{3} 
 \item{\hspace{-0.5em}}$_T$ Every countably generated model of $\Tpfree$  is isomorphic to a direct summand  of $\pfree$ that is at the same time an elementary substructure.

 \item\label{v}  Every countably generated module is purely embedded in  $\pfree$.
 \setcounter{enumi}{4} 
 \item{\hspace{-0.5em}}$_T$ Every countably generated model of $\Tpfree$   is elementarily embedded in $\pfree$.
 \item\label{vi} $R$ is left pure-semisimple. 
  %\item\label{(4ii)}  $R^{(\kappa)}$ is pure-universal among $\kappa$-generated flat $R$-modules and  $\Sigma$-pure-injective, for some (every) $\kappa\geq\omega$. 
%\item All flat $R$-modules are $\Sigma$-pure-injective.
%\item\label{(4iv)}   $T$ is totally transcendental.
%\item\label{(4i)}  $R$ is left $\Sigma$-pure-injective.
\end{enumerate}
Given an infinite cardinal $\kappa$, everywhere above `countably generated' can be replaced by `$\kappa$-generated' if at the same time $\pfree$ is replaced by $\pfree^{(\kappa)}$.

%\begin{enumerate}[\rm (i)]
%
% \item $R$ is left pure-semisimple, i.e., every left $R$-module is pure-injective, equivalently, pure-projective.
% \item Every (left) $R$-module is $\Sigma$-pure-injective (i.e., totally transcendental).
% \item Every  (left) $R$-module is Mittag-Leffler. \qed
%\end{enumerate}
 \end{cor}
 
\begin{proof}
 Only the extra clause needs proof. The dcc from (i) is  equivalent to all modules being ($\Sigma$-) pure-injective (or totally transcendental), \cite[Thm.4.5.1]{P2} or \cite[\S 11.1]{P1}. Then all pure-exact sequences split and  hence all modules are also pure-projective. It remains to apply Lemma \ref{lem} to see that now (i) implies all the other conditions for $\kappa$ as indicated. 
 
 For the converse, if any of those conditions hold for $\kappa$, it is not hard to see that they hold for $\aleph_0$ as well.
\end{proof}
% \begin{proof} (i) and (ii) are equivalent by definition. For (i) and (iii), take $\A=\Rmod$ in the theorem. Then $\C=\RMod$ and the $\Gamma_\B$-dcc is the dcc in the entire lattices of $n$-place pp formulas, for all $n$, which is known to be equivalent to pure-semisimplicty---by a result of Prest (predating his two books), \cite[Thm.\,4.5.1]{P2}. %Thus (i) and (iii) are equivalent. 
%% (i) and (ii) are equivalent by definition. As a module is $\Sigma$-pure-injective iff it has the dcc on pp subgroups (where unary formulas suffice
%\end{proof}

Next we apply the theorem to the other extreme, $\A= \{_RR\}$, the  free $R$-module of rank $1$. Then $\B$ consists of all f.g.\ projective modules, $\Add(\A) = \Add(\B)$ of all projective modules, and $\C$ of all flat modules. Remember, $F_{\aleph_0}$ denotes the free module of rank $\aleph_0$ and  $T_{\aleph_0}$ its complete theory.  The notion of perfect ring and the equivalence stated in (vi) below are due to Bass. %, \cite[Thm.P]{Ba}.

\begin{cor}\label{BassThm}
The following are equivalent for any ring $R$.

\begin{enumerate}[\rm (i)]%\label{thm}
\item\label{i} $\RMod$ has the dcc on all pp formulas of any arity $l(\bar x)$ of the form $A|\bar{x}$ where $A$ has $l(\bar x)$ many rows.% (equivalently, the dcc on all unary pp formulas).
\item\label{ii}  All flat modules are Mittag-Leffler.
 \setcounter{enumi}{1} 
 \item{\hspace{-0.5em}}$_T$ All flat models of $\Tfree$  are Mittag-Leffler.
\item\label{iii}  All countably generated flat modules  are projective.
 \setcounter{enumi}{2} 
 \item{\hspace{-0.5em}}$_T$ All countably generated flat models of $\Tfree$  are  projective.
  \item\label{iv}  Every countably generated flat module is isomorphic to a direct summand  of $\free$.
 \setcounter{enumi}{3} 
 \item{\hspace{-0.5em}}$_T$ Every countably generated flat model of $\Tfree$  is isomorphic to a direct summand of $\free$ that is at the same time an elementary substructure.

 \item\label{v}  Every countably generated flat module is purely embedded in  $\free$.
 \setcounter{enumi}{4} 
 \item{\hspace{-0.5em}}$_T$ Every countably generated flat model of $\Tfree$   is elementarily embedded in $\free$.
 \item\label{vi} $R$ is left perfect, i.e., all flat modules are projective. Equivalently, $R$ has the dcc on principal right ideals.
 %\item\label{(4ii)}  $R^{(\kappa)}$ is pure-universal among $\kappa$-generated flat $R$-modules and  $\Sigma$-pure-injective, for some (every) $\kappa\geq\omega$. 
%\item All flat $R$-modules are $\Sigma$-pure-injective.
%\item\label{(4iv)}   $T$ is totally transcendental.
%\item\label{(4i)}  $R$ is left $\Sigma$-pure-injective.
%
\end{enumerate}
Given an infinite cardinal $\kappa$, everywhere above `countably generated' can be replaced by `$\kappa$-generated' if at the same time $\free$ is replaced by $\free^{(\kappa)}$.
 \end{cor}
 
\begin{proof} (i). By \cite[Fact 2.8]{MPR}, the pp formulas that generate a pp type in a f.g.\ projective module are precisely the divisibility formulas $A|\bar{x}$, which is why (i) of the theorem turns into (i) as stated here. In (iii), pure-projective becomes projective, for flat+pure-projective=projective.

(vi). Bass introduced  perfect rings (in terms of perfect cover) and proved the two equivalent descriptions stated in (vi), \cite[Thm.P]{Ba}. %that they the notion  \cite[Thm.P]{Ba} proves the equivalence therein, while the In \cite{Ba} 
The dcc on principal right ideals is, in turn, equivalent to $_RR$ having the dcc on pp formulas of the form $a|x$ with $a\in R$.\footnote{T.\,Y.\,Lam says:
\emph{This switch
from left modules to right modules, albeit not new for Bass ..., is
in fact one of the inherent peculiar features of his Theorem...\,. Unfortunately,
because of this unusual switch of sides, [the] Theorem  is often misquoted in the
literature, sometimes even in authoritative sources;...\,\cite[p.\,24]{Lam}.} From the model-theoretic perspective (or in the terminology of p-functors of \cite{Zim}), there is nothing unusual about this switch. In fact, there is none, if we replace \emph{right} principal ideals by (left) pp formulas  that define them (or \emph{left} finitary matrix subgroups), which is the thing to do as our theorem suggests. The only switch of sides then is by the (rather accidental) fact that a \emph{left} pp formula defines a \emph{right} ideal---but that's in \emph{any} ring,  nothing  special about perfect.} So (i) implies (vi).

Conversely, assuming (vi), \emph{all} flat modules are projective, hence (vi) implies any of the conditions (ii) through (v), even in the `$\kappa$-generated' form of the extra clause.

Finally, as before, if any of those conditions hold for $\kappa>\aleph_0$,  they hold also for $\aleph_0$, which concludes the proof.
\end{proof}

 \begin{quest}\label{n=1}
As mentioned in Cor.\ \ref{Sim}(i), for pure-semisimplicity  it suffices to have the dcc just for  $n=1$. The same applies to left perfectness, see Cor.\ \ref{BassThm} (proof of) (vi), for apparently different reasons. This raises the question whether this is always true, i.e., true in Thm.\ \ref{M}(i). (Cf.\ the final remark.)
\end{quest}

\begin{quest}
 In both corollaries, the extra clause extends everything to $\kappa$-generated modules, because in both cases the following question has an affirmative answer---but again, for apparently different reasons. 
 
 Given $\A=\add\A\subseteq\Rmod$ containing $_RR$, if every module in $\C=\lim\A$ is Mittag-Leffler, is every module in $\C$  a direct sum of countably generated modules and thus  pure-projective?
%however we make sure it is a set, which is possibly by choosing a transversal, as $\Rmod$ is skeletally small, which is exactly that: the isomorphism classes of $\Rmod$ form a set. And so, throughout, we let $\A$ be a sub\emph{set} of $\Rmod$, 
\end{quest}

%\begin{conj}
% If every module in $\C=\lim\add\A$ is Mittag-Leffler, every module in $\C$ is a direct sum of countably generated modules and thus every module in $\C$ is pure-projective.
%\end{conj}
%\begin{thm}\label{Thm}
% 
%\end{thm}
We conclude with the application of the theorem to an intermediate class of f.p.\ modules, namely $\A_{cypr}$, the class of all \texttt{cyclically presented} modules, that is, the class of all modules of the form $R/Rr$ with $r\in R$. 

Set $\B_{cypr}:=\add \A_{cypr}$ and $C_{cypr}:=\lim B_{cypr}$. Write $F_{cypr}$ instead of $F_{\A_{cypr}}$ and $T_{cypr}$ instead of $T_{\A_{cypr}}$. Let further $\Gamma_{cypr}$ be the set of all finite sums of pp formulas of the form 
$\exists y(\bar x\dot= \bar a y\wedge ry\dot= 0)$  of any arity $l(\bar x)$, where $\bar a$ is an $l(\bar x)\times 1$ column vector over $R$.

A module is  \texttt{RD-projective} if it is a direct summand of a direct sum of
cyclically presented modules \cite[Cor.1]{War}. %This is not the standard definition, but is equivalent to
%it.
Clearly, the class of all of these is $\Add \A_{cypr} =\Add \B_{cypr}$.

%DEFINITION. If a short exact sequence satisfies the conditions of
%the previous proposition then A is said to be relatively divisible in B
%and the short exact sequence is RD-pure. The corresponding projectives and injectives will be referred to as RD-projectives and RDinjectives.
%This notion fits into the context of Proposition 1 if we (somewhat
%artificially) let RD be the class of cyclic modules isomorphic to modules
%of the form R/Rr. We can therefore derive the following corollary.
%COROLLARY 1. A left R-module is RD-projective if and only if
%it is a summand of a direct sum of cyclic modules of the form R/Rr
%(for various reR).
%C
%
%This is Warfield 1969 Purity and Alg Compactnes. ???Check it out???
  
\begin{cor} \label{Cor3} 
 The  following are equivalent for any ring $R$.

\begin{enumerate}[\rm (i)]%\label{thm}
\item\label{i} $\RMod$ has the $\Gamma_{cypr}$-dcc.
\item\label{ii} All modules in $\C_{cypr}$ are Mittag-Leffler.
 \setcounter{enumi}{1} 
 \item{\hspace{-0.5em}}$_T$ All $\C_{cypr}$-models of $T_{cypr}$  are Mittag-Leffler.
\item\label{iii}  All countably generated modules in   $\C_{cypr}$ are RD-projective.
 \setcounter{enumi}{2} 
 \item{\hspace{-0.5em}}$_T$ All countably generated $\C_{cypr}$-models of $T_{cypr}$  are  RD-projective.
  \item\label{iv}  Every countably generated module in   $\C_{cypr}$ is isomorphic to a direct summand of $F_{cypr}$.
 \setcounter{enumi}{3} 
 \item{\hspace{-0.5em}}$_T$ Every countably generated $\C_{cypr}$-model of $T_{cypr}$  is isomorphic to a direct summand of $F_{cypr}$ that is an elementary substructure of $F_{cypr}$.

 \item\label{v}  Every countably generated module in   $\C_{cypr}$ is purely embedded in $F_{cypr}$.
 \setcounter{enumi}{4} 
 \item{\hspace{-0.5em}}$_T$ Every countably generated $\C_{cypr}$-model of $T_{cypr}$   is elementarily embedded in $F_{cypr}$.
 %\item\label{(4ii)}  $R^{(\kappa)}$ is pure-universal among $\kappa$-generated flat $R$-modules and  $\Sigma$-pure-injective, for some (every) $\kappa\geq\omega$. 
%\item All flat $R$-modules are $\Sigma$-pure-injective.
%\item\label{(4iv)}   $T$ is totally transcendental.
%\item\label{(4i)}  $R$ is left $\Sigma$-pure-injective.
%
\end{enumerate}

\end{cor}

 We refrain from stating other analogous corollaries of the theorem and conclude with but one more example, one in which the corresponding dcc reduces to one on unary pp formulas, as discussed in Question \ref{n=1}.

\begin{rem}

\begin{enumerate}[\rm (a)]
 \item Let  $\A_{cyc}$ be the class of all cyclic f.p.\ modules, i.e., of modules of the form $R/I$ with $I$ a f.g.\ left ideal. For every $n>0$, let $\Gamma_n$ be the set of all finite sums of formulas of the form  $\exists y(\bar x\dot= \bar a y\wedge \bar b y\dot= 0)$ with $\bar a$ an $n \times 1$-column vector and $\bar b$ an arbitrary column vector over $R$. As is easily verified, this formula is freely realized by $\bar a$ in the module $R/I$, where $I$ is the left ideal generated by (the entries in) $\bar b$. Hence $\Gamma_n$ is the set of formulas that generate a pp $n$-type realized in a module from $\add \A_{cyc}$. Thus $\Gamma_{\A_{cyc}}=\bigcup_{n>0} \Gamma_n$. 

 \item Consider pp formulas $\phi(\bar x, \bar y)$ and  $\psi(\bar x, \bar y)$. Due to the additivity of pp formulas (as functors), one has $\psi(\bar x, \bar y)\leq \phi(\bar x, \bar y)$ if and only if $\psi(\bar 0, \bar y)\leq \phi(\bar 0, \bar y)$ and $\exists \bar y \psi(\bar x, \bar y)\leq \exists \bar y \phi(\bar x, \bar y)$. Therefore, a descending chain $\phi_0(\bar x, \bar y)\geq \phi_1(\bar x, \bar y)\geq \phi_2(\bar x, \bar y)\geq \ldots$ stabilizes if and only if the descending chains $\phi_0(\bar 0, \bar y)\geq \phi_1(\bar 0, \bar y)\geq \phi_2(\bar 0, \bar y)\geq \ldots$ and $ \exists \bar y\phi_0(\bar x, \bar y)\geq  \exists \bar y\phi_1(\bar x, \bar y)\geq  \exists \bar y\phi_2(\bar x, \bar y)\geq \ldots$ stabilize.

% consists of the finite sums of formulas of the form  %$\exists y(\bar x\dot= \bar a y\wedge \bar r y\dot= 0)$.
 \item Thus, if $\Gamma$ is a set of pp formulas that is closed under projections (i.e., if $\phi(\bar x, \bar y)\in \Gamma$, then $\exists \bar y \phi(\bar x, \bar y)\in \Gamma)$ and under kernels (i.e., if $\phi(\bar x, \bar y)\in \Gamma$, then $ \phi(\bar 0, \bar y)\in\Gamma$), then the $\Gamma$-dcc is equivalent to the dcc on unary formulas from $\Gamma$.
 
\item  If $\phi(\bar x, \bar y)$ generates the pp type of $(\bar a, \bar b)$ (in a certain module), then clearly $\exists \bar y \phi(\bar x, \bar y)$ generates the pp type of $\bar a$. Hence any set of pp formulas of the form $\Gamma_\G$ (cf.\ Notation \ref{gen}) is closed under arbitrary projections in the sense above. Therefore, to show that the $\Gamma_\G$-dcc reduces to the dcc on unary formulas from $\Gamma_\G$, it suffices to  verify closedness under `unary kernels', i.e., that if $\phi(x, \bar y)\in \Gamma_\G$, then also $\phi(0, \bar y)\in \Gamma_\G$.
\item $\Gamma_{\A_{cyc}}$ is so closed.  For, let $\phi(x, \bar y)$ be the formula $\exists z((x,\bar y)^t\dot= (a, \bar b)^t z\wedge \bar c z\dot= 0)$ (where $t$ stands for transpose), which is equivalent to 
$\exists z(x\dot=az \wedge \bar y\dot=\bar b z\wedge \bar c z\dot= 0)$. Its kernel is $\exists z(0\dot=az \wedge \bar y\dot=\bar b z\wedge \bar c z\dot= 0)$, which is indeed in $\Gamma_{\A_{cyc}}$. Namely, if $I$ is the left ideal generated by $\bar c$, i.e., $(a, \bar b)$ freely realizes $\phi(x, \bar y)$ in $R/I$, then  $\bar b$ freely realizes $\phi(0, \bar y)$ in $R/J$, where $J$ is the left ideal generated by $a$ and $\bar c$.

Consequently, in the notation of (a) above, the $\Gamma_1$-dcc implies the $\Gamma_{\A_{cyc}}$-dcc.

\item Finally, the equivalence \emph{(i)$\Leftrightarrow$(ii)} of the theorem can now be stated as follows (and similarly for the other items in the theorem).

$\RMod$ (equivalently $\Add({\A_{cyc}})$) has the $\Gamma_1$-dcc if and only if every direct limit of cyclic f.p.\ modules is Mittag-Leffler.

%its kernel $\phi(0, \bar y)$ is $\exists y(x\dot= a y\wedge \bar b y\dot= 0)$ is $\exists y(0\dot= a y\wedge \bar b y\dot= 0)$, which 

%\item ???A nicer one may be $R/I$ with $I$ finite intersections of principal left ideals (then one could reduce to unary?????), but when are those f.g.?
%  No, look at the formulas: . What are the free realizations???!!!

\end{enumerate}

\end{rem}

\noindent
Anand Pillay

Department of Mathematics University of Notre Dame, Notre Dame, IN 46556, USA

e-mail: Anand.Pillay.3@nd.edu\\

\noindent
Philipp Rothmaler

Department of Mathematics, The CUNY  Graduate Center, New York, NY 10016, USA

e-mail: philipp.rothmaler@bcc.cuny.edu

\end{document}